\documentclass[12pt,oneside]{amsart}
\usepackage{ae,aecompl}

\usepackage[english]{babel}

\usepackage{hyperref}

\usepackage{amsthm}
\usepackage{amssymb}

\numberwithin{equation}{section}
\numberwithin{figure}{section}

\theoremstyle{plain}
\newtheorem{thm}{Theorem}[section]
\newtheorem{lem}[thm]{Lemma}
\newtheorem{cor}[thm]{Corollary}
\newtheorem{prop}[thm]{Proposition}
\theoremstyle{definition}
\newtheorem{defn}[thm]{Definition}
\theoremstyle{remark}

\begin{document}

\title{Factorization of Matrices of Quaternions}

\author{Terry A. Loring}

\address{University of New Mexico, Department of Mathematics and Statistics,
Albuquerque, New Mexico, 87131, USA}

\subjclass{15B33}

\keywords{Kramers pair, matrix decompositions,
dual operation, quaternions}

\begin{abstract}
We review known factorization
results in quaternion matrices. Specifically, we derive the Jordan
canonical form, polar decomposition, singular value decomposition,
the QR factorization. We prove there is a Schur factorization for
commuting matrices, and from this derive the spectral theorem. We
do not consider algorithms, but do point to some of the numerical
literature. 

Rather than work directly with matrices of quaternions, we work
with complex matrices with a specific symmetry based on the dual
operation.  We discuss related results regarding complex matrices
that are self-dual or symmetric, but perhaps not Hermitian.

\end{abstract}

\maketitle

\section{The quaternionic condition}

It is possible to prove many factorization results for
matrices of quaterions by deriving them from their
familiar complex counterparts.  The amount of additional
work is surprisingly small.

There are more abstract factorization theorems that apply
to real $C^{*}$-algebras, along the lines of the delightful
paper by Pedersen on factorization in (complex)
$C^{*}$-algebras.  We are dealing here with more basic
questions, involving only finite-dimensional linear algebra.
These are never (hardly ever?) addressed in basis linear algebra
texts, creating the impression that linear algebra over the
quaternions is more difficult than it really is.

There are serious hazards within linear algebra over the
quaternions.  One can be lured to difficult questions of determinants
and handedness of the spectrum, leaving perhaps victorious, but
with the impression that all of linear algebra over the
quaternions is going to be difficult.

We mathematicians are well-advised, when upon such uneven ground,
to seek guidance from physicists.  Such guidance helped select
the topics, and helped suggest notation.   

One topic was included for reasons of elegance, the Jordan canonical
form.  For practical purposes, the Schur decomposition will
generally suffice.  We prove that as well.

It is assumed the reader is familiar with such things
as the spectral theorem and the functional calculus for
normal matrices.  It is not assumed that the reader knows
about $C^{*}$-algebras or physics, but these topics are
mentioned incidentally.

Let $\mathbb{H}$ denote the algebra of quaternions
\[
\mathbb{H}=\left\{ \left.a+b\hat{i}+c\hat{j}+d\hat{k}\right|a,b,c,d\in\mathbb{R}\right\} .
\]
This is an algebra over $\mathbb{R}$. The canonical embedding 
$ \mathbb{C}\hookrightarrow\mathbb{H} $
sending $1$ to $1$ and $i$ to $\hat{i}$ does not make $\mathbb{H}$
into an algebra over $\mathbb{C}$. The trouble is that the embedding
is not central. An additional algebraic operation is the involution
\[
\left(a+b\hat{i}+c\hat{j}+d\hat{k}\right)^{*}=a-b\hat{i}-c\hat{j}-d\hat{k}
\]
which satisfies several axioms, including the following.
\[
\alpha\in\mathbb{R},\ x\in\mathbb{H}\ \implies\ \left(\alpha x\right)^{*}=\alpha x^{*}
\]
\[
x^{*}x=0\ \implies\ x=0
\]
\[
\left(xy\right)^{*}=y^{*}x^{*}
\]

Turning to matrices, the algebra $\mathbf{M}_{N}(\mathbb{H})$ has
the expected structure of a unital $\mathbb{R}$-algebra, plus the
involution 
\[
\left[a_{ij}\right]^{*}=\left[a_{ji}^{*}\right]
\]
of conjugate-transpose. We would like to know
\[
A^{*}A=0\ \implies\ A=0
\]
and
\[
AB=I\ \implies\ BA=I.
\]
We will quickly prove these after we consider a representation of
$\mathbf{M}_{N}(\mathbb{H})$ on $\mathbb{C}^{2N}$.

We have an obvious representation of $\mathbf{M}_{N}(\mathbb{H})$
on $\mathbb{H}^{N}$ and quickly notice that since left and right
scalar multiplication of $\mathbb{H}^{N}$ disagree, we have both
left-eigenvalues $\lambda\in\mathbb{H}$ solving
\[
A\mathbf{v}=\lambda\mathbf{v}
\]
and right-eigenvalues $\mu\in\mathbb{H}$ solving
\[
A\mathbf{v}=\mathbf{v}\mu.
\]
A glance at the survey \cite{zhang1997quaternions} reveals that many
difficulties arise.

Ponder the situation for real matrices, in $\mathbf{M}_{N}(\mathbb{R})$.
To the chagrin of undergraduates, it is most natural to consider the
representation of $\mathbf{M}_{N}(\mathbb{R})$ on $\mathbb{C}^{N}$.
Given a real orthogonal matrix $O$ we get the full picture of why
it might not diagonalize in $\mathbf{M}_{N}(\mathbb{R})$ when we
look at the complex eigenvalues. It is at this point that we are implicitly
letting $\mathbf{M}_{N}(\mathbb{R})$ act on $\mathbb{C}^{N}$. 

We do well in the case of the quaternions to regard $\mathbf{M}_{N}(\mathbb{H})$
as represented on $\mathbb{C}^{2N}$, or in more modern terms that
de-emphasizes the role of vectors, via a certain embedding
\[
\chi:\mathbf{M}_{N}(\mathbb{H})\rightarrow\mathbf{M}_{2N}(\mathbb{C}).
\]
This is a very old trick. For $N>1$ it appears to have
been noticed first by H.~C.~Lee \cite{Lee_JordanFormQuaternions}. 

\begin{defn}
Given two \emph{complex} $N$-by-$N$ matrices $A$ and $B$ we set
\[
\chi\left(A+B\hat{j}\right)=\left[\begin{array}{cc}
A & B\\
-\overline{B} & \overline{A}\end{array}\right].
\]
\end{defn}

We generally study complex matrices in the image of $\chi$
and only at the end of our calculations do we draw conclusions about
\emph{$\mathbf{M}_{N}(\mathbb{H})$}. This is in keeping with applications
in quantum mechanics, where Hilbert space is always complex and time-reversal
symmetry will often be incorporated in the conjugate linear operation
$\mathcal{T}$.  We define
$\mathcal{T}:\mathbb{C}^{2N}\rightarrow\mathbb{C}^{2N}$
by
\begin{equation}
\mathcal{T}\left(\left[\begin{array}{c}
\mathbf{v}\\
\mathbf{w}\end{array}\right]\right)=\left[\begin{array}{c}
-\overline{\mathbf{w}}\\
\overline{\mathbf{v}}\end{array}\right].
\label{eq:time-reversal_anit-operator}
\end{equation}
See \cite[\S 2.2]{mehta2004random} for details on when $\mathcal{T}$
is relevant to time reversal symmetry.

A less generous description our our approach is we plan to study complex matrices
with a certain symmetry that is useful in physics, and then sell the
same results to pure mathematicians by re-branding them as theorems
about matrices of quaternions.

The operator $\mathcal{T}$ is relatively well behaved, despite
being only conjugate linear. It preserves orthogonality, and indeed 
\begin{equation}
\left\langle \mathcal{T}\xi,\mathcal{T}\eta\right\rangle 
=
\overline{\left\langle \xi,\eta\right\rangle }.
\label{eq:TRoperatorPreservesOrtho}
\end{equation}
Also
\begin{equation}
\mathbf{\xi}\perp\mathcal{T}\xi
\label{eq:TR-orthogonality}
\end{equation}
which is another one-line calculation.

\begin{lem}
The mapping $\chi$ is well-defined, is an $\mathbb{R}$-algebra homomorphism,
is one-to-one and satisfies
\[
\left(\chi\left(Y\right)\right)^{*}=\chi\left(Y^{*}\right).
\]
\end{lem}

\begin{proof}
Notice that every quaternion $q$ can be written as $\alpha+\beta\hat{j}$
with $\alpha$ and $\beta$ in $\mathbb{C}$. This makes the map well-defined.
It is clearly one-to-one and $\mathbb{R}$-linear. Notice $\beta\hat{j}=\hat{j}\bar{\beta}$
for complex number $\beta$ and so $B\hat{j}=\hat{j}\overline{B}$.
Therefore
\begin{align*}
\chi\left(A+B\hat{j}\right)\chi\left(C+D\hat{j}\right) & =\left[\begin{array}{cc}
A & B\\
-\overline{B} & \overline{A}\end{array}\right]\left[\begin{array}{cc}
C & D\\
-\overline{D} & \overline{C}\end{array}\right]\\
 & =\left[\begin{array}{cc}
AC-B\overline{D} & AD+B\overline{C}\\
-\overline{B}C-\overline{A}\overline{D} & -\overline{B}D+\overline{A}\overline{C}\end{array}\right]\\
 & =\chi\left(\left(AC-B\overline{D}\right)+\left(AD+B\overline{C}\right)\hat{j}\right)\\
 & =\chi\left(\left(A+B\hat{j}\right)\left(C+D\hat{j}\right)\right)
 \end{align*}
and 
\[
\left(\chi\left(A+B\hat{j}\right)\right)^{*}=\left[\begin{array}{cc}
A & B\\
-\overline{B} & \overline{A}\end{array}\right]^{*}
=\chi\left(A^{*}-B^{\mathrm{T}}\hat{j}\right)=\chi\left(\left(A+B\hat{j}\right)^{*}\right).
\]
\end{proof}

We can describe the image of $\chi$ in several useful ways. 
We will need several unary operations on complex matrices. For starters
we need the transpose $A^{\mathrm{T}}$ and the pointwise conjugate
$\overline{A}$ and the conjugate-transpose, or adjoint, 
$A^{*}=\overline{A}^{\mathrm{T}}=\overline{A^{\mathrm{T}}}$. Finally,
we need a twisted transpose that is useful in physics. This is a
``generalized involution'' $X^{\sharp}$ called the {\em dual operation}
that is defined only for $X$ in $\mathbf{M}_{2N}(\mathbb{C})$.

\begin{defn}
For $A,$ $B,$ $C$ and $D$ all complex $N$-by-$N$ matrices, define
\[
\left[\begin{array}{cc}
A & B\\
C & D\end{array}\right]^{\sharp}=\left[\begin{array}{cc}
D^{\mathrm{T}} & -B^{T}\\
-C^{\mathrm{T}} & A^{\mathrm{T}}\end{array}\right].
\]

Alternatively, we set $X^{\sharp}=-ZX^{\mathrm{T}}Z$ where
\begin{equation}
Z=\left[\begin{array}{cc}
0 & I\\
-I & 0\end{array}\right].
\label{eq:Define-Z}
\end{equation}
 
\end{defn}
Notice that $Z$ could be replaced by a matrix with similar properties
and define a variation on the dual operation. Indeed, the choice of
$Z$ is not standardized.

\begin{lem}
\label{lem:QuaternionicCondition}
For $X$ in $\mathbf{M}_{2N}(\mathbb{C})$
the following are equivalent.
\begin{enumerate}
\item $X$ is in the image of $\chi$;
\item $X^{*}=X^{\sharp}$;
\item $X=-\mathcal{T}\circ X\circ\mathcal{T}$ meaning
$X\xi=-\mathcal{T}\left(X\mathcal{T}\left(\xi\right)\right)$
for every vector $\xi$ in $\mathbb{C}^{2N}$ . 
\end{enumerate}
\end{lem}

\begin{proof}
Assume 
\[
X=\left[\begin{array}{cc}
A & B\\
-\overline{B} & \overline{A}\end{array}\right].
\]
 Then
 \[
X^{*}=\left[\begin{array}{cc}
\overline{A}^{\mathrm{T}} & -B^{\mathrm{T}}\\
\overline{B}^{\mathrm{T}} & A^{\mathrm{T}}\end{array}\right]=X^{\sharp}.
\]
Conversely $X^{*}=X^{\sharp}$ translates into block form as
\[
\left[\begin{array}{cc}
A^{*} & C^{*}\\
B^{*} & D^{*}\end{array}\right]=\left[\begin{array}{cc}
D^{\mathrm{T}} & -B^{\mathrm{T}}\\
-C^{\mathrm{T}} & A^{\mathrm{T}}\end{array}\right]
\]
so we have proven (1)$\iff$(2).

We compute $-\mathcal{T}\circ X\circ\mathcal{T}$ for $X$ again in
block form, and find
\begin{align*}
-\mathcal{T}\circ X\circ\mathcal{T}\left[\begin{array}{c}
\mathbf{v}\\
\mathbf{w}\end{array}\right] & =-\mathcal{T}\left(\left[\begin{array}{cc}
A & B\\
C & D\end{array}\right]\left[\begin{array}{c}
-\overline{\mathbf{w}}\\
\overline{\mathbf{v}}\end{array}\right]\right)\\
 & =-\mathcal{T}\left(\left[\begin{array}{c}
-A\overline{\mathbf{w}}+B\overline{\mathbf{v}}\\
-C\overline{\mathbf{w}}+D\overline{\mathbf{v}}\end{array}\right]\right)\\
 & =\left[\begin{array}{c}
-\overline{C}\mathbf{w}+\overline{D}\mathbf{v}\\
\overline{A}\mathbf{w}-\overline{B}\mathbf{v}\end{array}\right]\\
 & =\left[\begin{array}{cc}
\overline{D} & -\overline{C}\\
-\overline{B} & \overline{A}\end{array}\right]\left[\begin{array}{c}
\mathbf{v}\\
\mathbf{w}\end{array}\right]
\end{align*}
so the matrix for the linear operator $-\mathcal{T}\circ X\circ\mathcal{T}$
is $\left( X^\sharp \right) ^*$. Therefore (2)$\iff$(3). 
\end{proof}

\begin{defn}
Let us call
\[
X^{*}=X^{\sharp}
\]
the \emph{quaternionic condition} and the matrix $X$ we will
call a \emph{quaternionic matrix}.
\end{defn}

We now dispose of the implications
\[
A^{*}A=0\ \implies\ A=0
\]
and
\[
AB=I\ \implies\ BA=I
\]
for matrices of quaternions. These are true implications for complex
matrices, and in particular for quaternionic matrices, and therefore
true for matrices of quaternions.

We pause to note some axioms of the dual operation. It behaves a lot
like the transpose. It is linear,

\[
\left(X+\alpha Y\right)^{\sharp}=X+\alpha Y^{\sharp}
\]
which is true even for complex $\alpha$. Here $X$ and $Y$ are any
$2N$-by-$2N$ complex matrices. The dual reverses multiplication,
\[
\left(XY\right)^{\sharp}=Y^{\sharp}X^{\sharp}.
\]
It undoes itself,
\[
\left(X^{\sharp}\right)^{\sharp}=X
\]
and commutes with the adjoint
\[
\left(X^{\sharp}\right)^{*}=\left(X^{*}\right)^{\sharp}.
\]

\begin{lem}
Every matrix $G$ in $\mathbf{M}_{2N}(\mathbb{C})$ can we expressed
in a unique way as 
\[
G=X+iY
\]
with $X$ and $Y$ being quaternionic matrices. 
\end{lem}

\begin{proof}
Given $G$ we set
\[
X=\frac{1}{2}G^{\sharp*}+\frac{1}{2}G
\]
and 
\[
Y=\frac{i}{2}G^{\sharp*}-\frac{i}{2}G.
\]
\end{proof}

\section{Kramers degeneracy and Schur factorization}

Eigenvalue doubling is a key feature of self-dual, self-adjoint matrices.
In physics this is called the Kramers degeneracy theorem, or the theory
of Kramers pairs \cite{kramers1930theorie,wigner1932uber,mehta2004random}.
This generalizes in two ways, to give eigenvalue doubling given the symmetry
$X=X^{\sharp}$ and conjugate-pairing of eigenvalues given the
symmetry $X^{*}=X^{\sharp}$.

Such a collection of paired eigenvectors will, in good situations, form a unitary
matrix. If $U$ is unitary matrix that satisfies the quaternionic
condition then it satisfies another symmetry making it symplectic,
specifically $U^{\mathrm{T}}ZU=Z$.

\begin{lem}
\label{lem:SymplecticUnitaryStructure} 
Suppose U is a unitary $2N$-by-$2N$ matrix. The following are equivalent:
\begin{enumerate}
\item $U$ is symplectic;
\item $U^{*}=U^{\sharp}$;
\item $Z\overline{U}=UZ$;
\item $U\circ\mathcal{T}=\mathcal{T}\circ U$; 
\item If $\mathbf{v}$ is column $j$ of $U$ for $j\leq N$ then column
$N+j$ of $U$ is $\mathcal{T}\left(v\right)$. 
\end{enumerate}
\end{lem}

\begin{proof}
This follows easily from Lemma~\ref{lem:QuaternionicCondition}.
\end{proof}

We need to know that the group of symplectic unitary acts transitively
on $\mathbb{C}^{n}$.

\begin{lem}
\label{lem:transitivity} 
If $\mathbf{v}$ is any unit vector in $\mathbb{C}^{2N}$
then there is a symplectic unitary with $U\mathbf{e}_{1}=\mathbf{v}$. 
\end{lem}

\begin{proof}
Let $\mathbf{v}_{1}=\mathbf{v}$. We use (\ref{eq:TRoperatorPreservesOrtho})
and (\ref{eq:TR-orthogonality}) to select in order vectors that are
an orthonormal basis for $\mathbb{C}^{2N}$, but of the form
\[
\mathbf{v}_{1},\mathcal{T}\mathbf{v}_{1},\mathbf{v}_{2},\mathcal{T}\mathbf{v}_{2},\dots,\mathbf{v}_{N},\mathcal{T}\mathbf{v}_{N}.
\]
If we reorder to 
\[
\mathbf{v}_{1},\mathbf{v}_{2},\dots,\mathbf{v}_{N},\mathcal{T}\mathbf{v}_{1},\mathcal{T}\mathbf{v}_{2},\dots,\mathcal{T}\mathbf{v}_{N}
\]
we have the columns of the desired symplectic unitary. 
\end{proof}

The most basic result in this realm is an ugly lemma that says that
$\mathcal{T}$ maps left eigenvectors of $X$ to right eigenvectors
of $X^{\sharp}$, with the same eigenvalue. 
The second part of this lemma is more elegant, if less general.  It
specifices how every quaternionic matrix has a conjugate symmetry it
its spectral decomposition.

\begin{lem}
\label{lem:left-right-Lemma}
Suppose $X$ is in $\mathbf{M}_{2N}(\mathbb{C})$.
\begin{enumerate}
\item If $X\xi=\lambda\xi$ then 
\[
\left(\mathcal{T}\xi\right)^{*}X^{\sharp}
=
\lambda\left(\mathcal{T}\xi\right)^{*}.
\]

\item If $X^{*}=X^{\sharp}$ and $X\xi=\lambda\xi$ then 
\[
X\left(\mathcal{T}\xi\right)
=
\overline{\lambda}\left(\mathcal{T}\xi\right).
\]
\end{enumerate}
\end{lem}

\begin{proof}
(1) Starting with 
\[
X=\left[\begin{array}{cc}
A & B\\
C & D\end{array}\right]
\]
and
\[
\xi=\left[\begin{array}{c}
\mathbf{v}\\
\mathbf{w}\end{array}\right]
\]
we find that $X\xi=\lambda\xi$ translates to 
\begin{align*}
A\mathbf{v}+B\mathbf{w} & =\lambda\mathbf{v}\\
C\mathbf{v}+D\mathbf{w} & =\lambda\mathbf{w}
\end{align*}
and $\left(\mathcal{T}\mathbf{v}\right)^{*}X^{\sharp}=\lambda\left(\mathcal{T}\mathbf{v}\right)^{*}$
translates to 
\begin{align*}
-\mathbf{w}^{\mathrm{T}}D^{\mathrm{T}}-\mathbf{v}^{\mathrm{T}}C^{\mathrm{T}} & =-\lambda\mathbf{w}^{\mathrm{T}}\\
\mathbf{w}^{\mathrm{T}}B^{\mathrm{T}}+\mathbf{v}^{\mathrm{T}}A^{\mathrm{T}} & =\lambda\mathbf{v}^{\mathrm{T}}
\end{align*}
so these are equivalent conditions.

(2) follow from (1) by taking adjoints.
\end{proof}

Now we are able to extend Kramers degeneracy to a variety of situations,
starting with a block diagonalization for commuting quaternionic matrices.
Part (2) of Theorem~\ref{thm:QuaternionicSchur} appeared
in \cite{wiegmannMatricesQuaternion}.

\begin{thm}
\label{thm:QuaternionicSchur}
Suppose $X_{1},\ldots,X_{k}$ in $\mathbf{M}_{2N}(\mathbb{C})$
commute pairwise and $X_{j}^{*}=X_{j}^{\sharp}$ for all $j$. 
\begin{enumerate}
\item There is a single symplectic unitary $U$ so that, for all $j,$
\[
U^{*}X_{j}U=\left[\begin{array}{cc}
T_{j} & S_{j}\\
-\overline{S_{j}} & \overline{T_{j}}\end{array}\right].
\]
with $T_{j}$ upper-triangular and $S_{j}$ strictly upper-triangular. 
\item If, in addition, the $X_{j}$ are normal then there is a single symplectic
unitary $U$ so that, for all $j,$
\[
U^{*}X_{j}U=\left[\begin{array}{cc}
D_{j} & 0\\
0 & \overline{D_{j}}\end{array}\right].
\]
 with $D_{j}$ diagonal.
\item Every symplectic unitary has determinant one.
\end{enumerate}
\end{thm}

\begin{proof}
(1) A finite set of commuting matrices will have a common eigenvector,
so let $\mathbf{v}$ be a unit vector so that $X_{j}\mathbf{v}=\lambda_{j}\mathbf{v}.$
We know then that so $\mathcal{T}\mathbf{v}$ is an eigenvalue for
$X_{j}$ with eigenvector $\overline{\lambda_{j}}.$ There is a symplectic
unitary $U_{1}$ so that $U_{1}\mathbf{e}_{1}=\mathbf{v},$ and then
$U_{1}\mathbf{e}_{N+1}=\mathcal{T}\mathbf{v}$ by
Lemma~\ref{lem:SymplecticUnitaryStructure}.
Let $Y_{j}=U_{1}^{*}X_{j}U_{1}.$ Then 
\[
Y_{j}\mathbf{e}_{1}=\lambda_{j}\mathbf{e}_{1}
\]
and
\[
Y_{j}\mathbf{e}_{N+1}=\overline{\lambda_{j}}\mathbf{e}_{N+1}.
\]
 This means the column $1$ and column $N+1$ are all but zeroed-out,
\[
Y_{j}=\left[\begin{array}{cccc}
\lambda_{j} & * & 0 & *\\
0 & A_{j} & 0 & C_{j}\\
0 & * & \overline{\lambda_{j}} & *\\
0 & B_{j} & 0 & D_{j}\end{array}\right].
\]
Up to a non-symplectic change of basis we are looking at block-upper
triangular matrices that commute, so the lower-right corners in that
basis commute. This means that the 
\begin{equation}\label{eqn:smallerBlocks}
Z_{j}=\left[\begin{array}{cc}
A_{j} & C_{j}\\
B_{j} & D_{j}\end{array}\right]
\end{equation}
all commute, and each must satisfy $Z_{j}^{*}=Z_{j}^{\sharp}.$ By
induction, we have proven the first claim.

For (2) we modify the proof just a little. Starting with the $X_{j}$
normal, we find the $Y_{j}$ are also normal and so 
\[
Y_{j}=\left[\begin{array}{cccc}
\lambda_{j} & 0 & 0 & 0\\
0 & A_{j} & 0 & C_{j}\\
0 & 0 & \overline{\lambda_{j}} & 0\\
0 & B_{j} & 0 & D_{j}\end{array}\right].
\]
This, after an appropiate basis change, would be block-diagonal,
and from this we can conclude that the matrix in (\ref{eqn:smallerBlocks})
is normal.  The induction proceeds as before, with the stronger conclusion
that $B_j = C_j=0$  and $A_j = \overline{D_j}$  is diagonal.

(3) Applying (2) we find
\[
W=U\left[\begin{array}{cc}
D & 0\\
0 & \overline{D}\end{array}\right]U^{*}
\]
where $D$ is a diagonal unitary. Therefore
\[
\det(W)=\det(D)\det(\overline{D})=\det(D)\overline{\det(D)} \geq 0
\]
and since a unitary has determinant on the unit circle, we are done.
\end{proof}

\begin{cor}
Every matrix in $\mathbf{M}_{N}(\mathbb{H})$ is unitarily equivalent
to a upper-triangular matrix in $\mathbf{M}_{N}(\mathbb{H})$ that
has complex numbers on the diagonal.
\end{cor}

There is an algorithm \cite{bunse1989quaternion} for the Schur decomposition
of quaternionic matrices. 

\begin{cor}
\label{cor:cojugetToComplex}
Every normal matrix in $\mathbf{M}_{N}(\mathbb{H})$
is unitarily equivalent to a diagonal matrix in $\mathbf{M}_{N}(\mathbb{C})$.
Every Hermitian matrix in $\mathbf{M}_{N}(\mathbb{H})$ is unitarily
equivalent to a diagonal matrix in $\mathbf{M}_{N}(\mathbb{R})$.
\end{cor}

\begin{cor}
Every Hermitian self-dual matrix $X$ in $\mathbf{M}_{2N}(\mathbb{C})$
is of the form
\[
X=U\left[\begin{array}{cc}
D_{j} & 0\\
0 & D_{j}\end{array}\right]U^{*}
\]
for some symplectic unitary $U$ and a diagonal real matrix $D_{j}$.
\end{cor}

These corollaries cause us to reconsider the concept of left and right
eigenvalues in $\mathbb{H}$ and focus on just those that are in $\mathbb{C}$.
For details on left eigenvalues and non-complex right eigenvalues,
consult \cite{FarenickSpectralThmquatern}. We put the complex right
eigenvalues in a simple context with the following. 

\begin{lem}
Suppose $\mathbf{v},\mathbf{w}\in\mathbb{C}^{N}$ and $\lambda\in\mathbb{C}$
and $A,B\in\mathbf{M}_{N}(\mathbb{C})$. Then
\[
\left[\begin{array}{cc}
A & B\\
-\overline{B} & \overline{A}\end{array}\right]\left[\begin{array}{c}
\mathbf{v}\\
\mathbf{w}\end{array}\right]=\lambda\left[\begin{array}{c}
\mathbf{v}\\
\mathbf{w}\end{array}\right]
\]
if and only if
\[
\left(A+B\hat{j}\right)\left(\mathbf{v}-\hat{j}\mathbf{w}\right)=\left(\mathbf{v}-\hat{j}\mathbf{w}\right)\lambda.
\]
Therefore $\lambda\in\mathbb{C}$ is a right eigenvalue of $X\in\mathbf{M}_{N}(\mathbb{H})$
if and only if $\lambda$ is an eigenvalue of $\chi\left(X\right)$.
\end{lem}

\begin{proof}
This a short, direct calculation.
\end{proof}

\section{Jordan canonical form}

Kramers degeneracy extends to the generalized eigenvectors used to
find the Jordan canonical form.

\begin{lem}
Suppose $X^\sharp = X^*$.  If 
\[
\left( X - \lambda I \right) ^r \mathbf v = 0  
\quad \mbox{and}  \quad
\left( X - \lambda I \right) ^{r-1} \mathbf v  \neq 0 
\]
then
\[
\left( X - \overline{\lambda} I \right) ^r \mathcal T \mathbf v = 0 
\quad \mbox{and}  \quad
\left( X - \overline{\lambda} I \right) ^{r-1} \mathcal T \mathbf v \neq 0. 
\]
\end{lem}

\begin{proof}
For any $Y$  we have  
\[
\| Y \mathbf v \|
= \| \overline{Y} \overline{\mathbf v} \| 
= \| Z \overline{Y} ZZ\overline{\mathbf v} \| 
\]
proving
\[
\| Y \mathbf v \|  =  \| Y ^{\sharp *} \mathcal T \mathbf v \| .
\]
Since
\[
\left( \left(  X -  \lambda I \right) ^k \right)^{\sharp *}
= \left(  X - \overline{\lambda} I \right) ^k
\]
the result follows.
\end{proof}

We see the general idea of a proof \cite{zhang2001jordan} of the Jordan decomposition for
a quaternionic matrix.  When we build a Jordan basis we need to respect
$ \mathcal T $ in two ways.  Whatever basis we pick for the subspace
corresponding to $\lambda$ with positive imaginary part, we apply
$ \mathcal T $ to get the basis for the subspace
corresponding to $\overline{\lambda}$.  When $\lambda $ is real we need to
pick generalized eigenvectors in pairs.

The Jordan form of a quaternionic matrix is not so elegant, as each
Jordan blocks larger than $2$-by-$2$ gets spread around to all four 
quadrants of the matrix.  We work directly with a Jordan basis and
then make our final conclusion in terms of quaternions.

\begin{thm}
Suppose $X^\sharp = X^*$. There is a Jordan basis for $X$ consisting
of pairs of the form
$ \mathbf v, \mathcal T \mathbf v$.
\end{thm}

\begin{proof}
Let $N_\lambda$ denote the subspace of all generalized eigenvectors
for $\lambda$ toghether with the zero vector. 
Recall that just treating $X$ as a complex matrix, the procedure to
select a Jordan basis involves selecting, for each $\lambda$, a
basis for $N_\lambda$ with the following property:  whenever
$\mathbf b$ is in this basis, then $(X - \lambda I)\mathbf b$
is either zero or back in this basis.

For $\lambda$ in the spectrum with positive imaginary part we make
such a choice and then apply $\mathcal T $ to get 
a set of vectors in $N_{\overline{\lambda}}$ that has the correct
number of elements to be a basis of $N_{\overline{\lambda}}$.
It will also be a linearly independent set since 
$Z$ and conjugation both preserve linear independence.  Since 
\[
(X - \overline{\lambda} I) \mathcal T \mathbf b
= \mathcal T (X -  \lambda  I) \mathbf b
\]
this basis of  $N_{\overline{\lambda}}$ has the desired property.

For $\lambda$ in the spectrum that is real we need to modify the
procedure for selecting the basis of $N_ \lambda $.  A common
procedure selects a basis $\mathbf{b}_{r,1},\dots,\mathbf{b}_{r,m_{r}}$
for
\begin{equation}
\ker\left(X-\lambda\right)^{r}
\cap \left( \ker\left(X-\lambda\right)^{r-1} \right) ^{\perp}
\cap \left( \mathrm{im}\left(X-\lambda\right) \right) ^{\perp}
\label{eq:endOfChains}
\end{equation}
and constructs for $N_{\lambda}$  the Jordan basis 
\[
\left\{
\left(X-\lambda\right)^{j}\mathbf{b}_{r,k}
\left|\strut\,
1\leq r\leq r_{\mathrm{max}},\ 
0\leq j\leq r-1,\ k=1,\dots,m_{r}
\right.
\right\}
\]
Since
\[
\mathcal{T}\left(X-\lambda\right)^{j}\mathbf{b}_{r,k}
=
\left(X-\lambda\right)^{j}\mathcal{T}\mathbf{b}_{r,k}
\]
we get the desired structure if the subspaces
(\ref{eq:endOfChains}) are $\mathcal{T}$-invariant. This follows
from the next two lemmas and the equality 
\[
\left( \mathrm{im}\left(X-\lambda\right) \right) ^{\perp}
=
\ker\left(X^*-\overline{\lambda}\right).
\]
We can now assemble the bases of the various $N_{\lambda}$ to get
a Jordan basis built from the Kramers pairs.
\end{proof}

\begin{lem}
If $X^\sharp = X^*$  then $\ker (X)$ is $\mathcal{T}$-invariant.
\end{lem}

\begin{proof}
This is a restatement of the $\lambda=0$ case of Lemma~\ref{lem:left-right-Lemma}(2).
\end{proof}

\begin{lem}
If a subspace $\mathcal H _0 $ is $\mathcal{T}$-invariant
then  $\left(\mathcal H _0 \right) ^\perp $ is also
$\mathcal{T}$-invariant.
\end{lem}

\begin{proof}
If $\mathbf v$ is in $\left(\mathcal H _0 \right) ^\perp $
then for every $\mathbf w \in \mathcal H _0$ we have
\[
\left\langle \mathcal{T}\mathbf{v},\mathbf{w}\right\rangle 
=-\left\langle \mathcal{T}\mathbf{v},\mathcal{T}\mathcal{T}\mathbf{w}\right\rangle 
=-\overline{\left\langle \mathbf{v},\mathcal{T}\mathbf{w}\right\rangle }
=0 .
\]

\end{proof}

\begin{cor}
Suppose $X \in \mathbf M _{n} ( \mathbb H )$. There is an invertible
matrix $S \in \mathbf M _{n} ( \mathbb H )$ and a complex matrix $J$ in
Jordan Form such that $X = S^{-1} J S $.
\end{cor}

\section{Norms }

There are two operator norms to consider on $X$ in $\mathbf{M}_{N}(\mathbb{H})$,
that induced by quaternionic Hilbert space and that induced by complex
Hilbert space on $\chi\left(X\right)$. They end up identical.
\begin{thm}
Suppose $X$ is in $\mathbf{M}_{N}(\mathbb{H})$. Then using the norms
\[
\left\Vert \mathbf{v}\right\Vert =\left(\sum_{j=1}^{N}v_{j}^{*}v_{j}\right)^{\frac{1}{2}}
\]
on $\mathbb{H}^{N}$ and
\[
\left\Vert \mathbf{w}\right\Vert =\left(\sum_{j=1}^{2N}\overline{v_{j}}v_{j}\right)^{\frac{1}{2}}
\]
on $\mathbb{C}^{2N}$ we have 
\[
\sup_{\mathbf{v}\neq0}\frac{\left\Vert X\mathbf{v}\right\Vert }{\left\Vert \mathbf{v}\right\Vert }
=\sup_{\mathbf{w}\neq0}\frac{\left\Vert \chi\left(X\right)\mathbf{w}\right\Vert }{\left\Vert \mathbf{w}\right\Vert }.
\]
\end{thm}

\begin{proof}
Utilizing also the norm on $\mathbb{C}^{2N}$ we calculate the four
relevant norms: 
\[
\left\Vert \left[\begin{array}{c}
\mathbf{v}\\
\mathbf{w}\end{array}\right]\right\Vert ^{2}
=
\left\Vert \mathbf{v}\right\Vert ^{2}+\left\Vert \mathbf{w}\right\Vert ^{2}.
\]
\[
\left\Vert \mathbf{v}-\hat{j}\mathbf{w}\right\Vert ^{2}
=
\left\Vert \mathbf{v}\right\Vert ^{2}+\left\Vert \mathbf{w}\right\Vert ^{2}.
\]
\[
\left\Vert \left[\begin{array}{cc}
A & B\\
-\overline{B} & \overline{A}\end{array}\right]\left[\begin{array}{c}
\mathbf{v}\\
\mathbf{w}\end{array}\right]\right\Vert ^{2}
=
\left\Vert A\mathbf{v}+B\mathbf{w}\right\Vert ^{2}+\left\Vert \overline{A}\mathbf{w}-\overline{B}\mathbf{v}\right\Vert ^{2}.
\]
\begin{align*}
\left\Vert \left(A+B\hat{j}\right)\left(\mathbf{v}-\hat{j}\mathbf{w}\right)\right\Vert ^{2} 
 & =\left\Vert A\mathbf{v}+B\mathbf{w}+\hat{j}\left(\overline{B}\mathbf{v}-\overline{A}\mathbf{w}\right)\right\Vert ^{2}\\
 & =\left\Vert A\mathbf{v}+B\mathbf{w}\right\Vert ^{2}+\left\Vert \overline{B}\mathbf{v}-\overline{A}\mathbf{w}\right\Vert ^{2}.
\end{align*}
The result follows.
\end{proof}

\section{Singular value decomposition}

As in the complex case, a slick way to prove there is a singular value
decomposition is to work out the polar decomposition and then use
the spectral theorem on the positive part.

In \cite{zhang1997quaternions} it is stated that  ``a little more work is
needed for the singular case'' when discussing the polar decomposition.
The extra work involves padding out a quaternionic partial isometry
to be a quaternionic unitary (so symplectic unitary.)

We remind the reader that $U$  is a partial isometry when
$(U^*U)^2 = U^*U$, or equivalently $UU^*U = U$ or $U^*UU^* = U^*$
or $(UU^*)^2 = UU^*$.  If we restrict the domain and range of $U$
we find it is an isometry from  $\left(\ker(U)\right)^{\perp}$
to  $\left(\ker(U^*)\right)^{\perp}$.

\begin{lem}
\label{lem:IsometriesExtendQuaterionically}
Suppose $U^{*}=U^{\sharp}$and
$U$ is a partial isometry in $\mathbf{M}_{2N}(\mathbb{C})$. There
is a symplectic unitary $W$ in $\mathbf{M}_{N}(\mathbb{C})$ so that
$W\xi=U\xi$ for all $\xi\perp\ker(U)$.
\end{lem}

\begin{proof}
If $\mathbf{v}$ is in $\ker(U)$ then by Lemma~\ref{lem:left-right-Lemma},
$\mathcal{T}\mathbf{v}$ is also in $\ker(U)$. Since $\mathbf{v}$
and $\mathcal{T}\mathbf{v}$ are orthogonal, we can show that $\ker(U)$
has even dimension $2m$ and that is has a basis for the form 
\[
\mathbf{v}_{1},\dots,\mathbf{v}_{m},\mathcal{T}\mathbf{v}_{1},\dots,\mathcal{T}\mathbf{v}_{m}.
\]
We are working in finite dimensions so the dimension of $\ker(U^{*})$
is also $2m$ and we select for it a basis 
\[
\mathbf{w}_{1},\dots,\mathbf{w}_{m},\mathcal{T}\mathbf{w}_{1},\dots,\mathcal{T}\mathbf{w}_{m}.
\]
We can define $W$ to agree with $U$ on $\left(\ker(U)\right)^{\perp}$
and to send $\mathbf{v}_{j}$ to $\mathbf{w}_{j}$ and $\mathcal{T}\mathbf{v}_{j}$
to $\mathcal{T}\mathbf{w}_{j}$ and so get a unitary that commutes
with $\mathcal{T},$ which means it is symplectic.
\end{proof}

\begin{lem}
Suppose $X^{*}=X^{\sharp}$ in $\mathbf{M}_{2N}(\mathbb{C})$. Then
there is a unitary $U$ and a positive semidefinite $P$ with $U^{*}=U^{\sharp}$and
$P^{*}=P^{\sharp}$ and $X=UP$.
\end{lem}

\begin{proof}
Let $f$ be a continuous function on the positive reals with $f(0)=0$
and $f(\lambda)=\lambda^{-\frac{1}{2}}$ for every nonzero eigenvalue
of $X^{*}X$. Then let
\[
W=Xf(X^{*}X).
\]
The usual calculations in functional calculus tell us $W$ is a partial
isometry and that $X=WP$ for $P=\left(X^{*}X\right)^{\frac{1}{2}}$.
Working with monomials, polynomials and then taking limits, we can
show
\[
\left(f(Y)\right)^{\sharp}=f\left(Y^{\sharp}\right)
\]
 for any positive operator $Y$ and so 
 \[
W^{\sharp}=f\left(X^{\sharp}X^{*\sharp}\right)X^{\sharp}=f\left(X^{*}X\right)X^{*}=W^{*}.\]
Also $P^{\sharp}=P=P^{*}$.

We just showed that the matrices in the minimal polar decomposition
are quaternionic. We use Lemma~\ref{lem:IsometriesExtendQuaterionically}
to finish the argument.
\end{proof}

As expected, the polar decomposition leads to a singular value decomposition.

\begin{thm}
Suppose $X^{*}=X^{\sharp}$ in $\mathbf{M}_{2N}(\mathbb{C})$. There
are symplectic unitary matrices $U$ and $V$  and a diagonal matrix
$D$ with nonnegative real entries and $D^{\sharp}=D^{*}$ so that
$X=UDV$.
\end{thm}

\begin{proof}
We take a quaternionic polar decomposition $X=WP$. Since $P$ is
positive, we apply Theorem~\ref{thm:QuaternionicSchur} to get symplectic
unitary matrices $Q$ and $V$ and diagonal matrix $D$ so that $P=QDV$.
The eigenvalues of $P$ are nonnegative, so the same is true for the
diagonal elements of $D$ and we have the needed factorization $X=(WQ)DV$. 
\end{proof}

\section{QR factorization}

It is easy to use Lemma~\ref{lem:transitivity} to get over $\mathbb{H}$
a QR factorization theorem. Notice that upper triangular matrices are
sent by $\chi$ to matrices that are block upper-triangular.

\begin{thm}
If $X^{*}=X^{\sharp}$ in $\mathbf{M}_{2N}(\mathbb{C})$ the there
is a symplectic unitary $Q$ and $R$ of the form
\[
R=\left[\begin{array}{cc}
A & B\\
-\overline{B} & \overline{A}\end{array}\right]
\]
with $A$ and $B$ upper triangular. If $X\in\mathbf{M}_{N}(\mathbb{H})$
then there is a untiary $Q$ and upper triangular matrix R in $\mathbf{M}_{N}(\mathbb{H})$
so that $X=QR$.
\end{thm}

\begin{proof}
(2) follows directly from (1), so we prove (1).

We apply Lemma ~\ref{lem:transitivity} to the first column of $X$
and find $ X=Q_{1}R_{1} $ where 
\[
R_{1}=\left[\begin{array}{cc}
A_{1} & B_{1}\\
-\overline{B_{1}} & \overline{A_{1}}\end{array}\right]
\]
where $A_{1}$ and $B_{1}$ have zeros in their first columns, except
perhaps in the top position. As we did earlier, we can proceed with
a proof by induction.
\end{proof}

\section{Self-dual matrices}

If we study matrices with $X^{\sharp}=X$ then we are no longer working
directly with quaternionic matrices, but as we discuss below, there
is a connection. We discuss a Schur factorization and a structured
polar decomposition for self-dual matrices. The latter is a bit tricky,
so we warm up with a structured polar decomposition for symmetric
complex matrices. 

For dealing with a single self-dual matrix, there is the efficient
Paige / Van Loan algorithm \cite{HastLorTheoryPractice,paigeVanLoan} to implement
the following theorem.

\begin{thm}
\label{thm:self-dual-shur}
Given $X_{1},\ldots,X_{k}$ in $\mathbf{M}_{2N}(\mathbb{C})$
that commute pairwise and are self-dual, there is a single symplectic
unitary $U$ so that for all $j,$
\[
U^{*}X_{j}U=\left[\begin{array}{cc}
T_{j} & C_{j}\\
0 & T_{j}^{\mathrm{T}}\end{array}\right].
\]
with $T_{j}$ upper-triangular and the \textbf{$C_{j}$ }skew-symmetric.
\end{thm}

\begin{proof}
Let $\mathbf{v}$ be a nonzero unit vector so that
$X_{j}\mathbf{v}=\lambda_{j}\mathbf{v}$
for all $j$. By Lemma~\ref{lem:left-right-Lemma},
\[
X_{j}^{*}\left(\mathcal{T}\mathbf{v}\right)
=\overline{\lambda_{j}}\left(\mathcal{T}\mathbf{v}\right).
\]
There is a symplectic unitary $U_{1}$ so that $U_{1}\mathbf{e}_{1}=\mathbf{v}$
and $U_{1}\mathbf{e}_{N+1}=\mathcal{T}\mathbf{v}.$ Let $Y_{j}=U_{1}^{*}X_{j}U_{1}.$
Then 
\[
Y_{j}\mathbf{e}_{1}=\lambda_{j}\mathbf{e}_{1}
\]
and
\[
Y_{j}^{*}\mathbf{e}_{N+1}=\overline{\lambda_{j}}\mathbf{e}_{N+1}.
\]
Since $\mathbf{e}_{N+1}$ is real, we take adjoint and discover
\[
\mathbf{e}_{N+1}^{\mathrm{T}}Y_{j}=\lambda_{j}\mathbf{e}_{N+1}^{\mathrm{T}}.
\]
This means the column $1$ and row $N+1$ are all but zeroed-out,
\[
Y_{j}=\left[\begin{array}{cccc}
\lambda_{j} & * & * & *\\
0 & A_{j} & * & C_{j}\\
0 & 0 & \lambda_{j} & 0\\
0 & B_{j} & * & D_{j}\end{array}\right].
\]
Basic facts about block triangular matrices show that the 
\[
Z_{j}=\left[\begin{array}{cc}
A_{j} & C_{j}\\
B_{j} & D_{j}\end{array}\right]
\]
are a commuting family of matrices, and since $U_{1}$ was chosen
to be symplectic, the $Z_{j}$ will be self-dual. As simple induction
now finishes the proof.
\end{proof}

A promising numerical technique for the joint diagonalization of two
commuting self-dual self-adjoint matrices $H$ and $K$ would be to
form the normal self-dual matrix $X=A+iB$ and apply Paige / Van Loan
to reduce to block diagonal form. Then apply ordinary Schur decomposition.
This technique was used in \cite[\S 9]{HastLorTheoryPractice} to
diagonalize matrices that were exactly self-dual and approximately
unitary. This idea is mentioned in \cite{bunse1993numerical}, section
6.5. 

It is not hard to show that the minimal polar decomposition, the one
that is unique and can involve a partial isometry, preserves in some
way just about any symmetry thrown at it. This is because the functional
calculus interacts well with the dual operation \cite{LoringSorensenRealLins},
as well as with the transpose. It is a bit harder to figure what happens
for the maximal polar decomposition. By the minimal polar decomposition
is meant the factorization that is unique and can involve a partial
isometry. By the maximal polar decomposition is meant the factorization
that involves a unitary.

We begin with the easier result about the polar decomposition of complex
symmetric matrices.

\begin{thm}
\label{thm:SymmetricPolar} 
If $X$ in $\mathbf{M}_{2n}(\mathbb{C})$
satisfies $X^{\mathrm{T}}=X$ then there is a unitary $U$ so that
$U^{\mathrm{T}}=U$ and
$ X=U\left|X\right|$.
\end{thm}

\begin{proof}
Again chose $f$ with $f(0)=0$ and $f(\lambda)=\lambda^{-\frac{1}{2}}$
for every nonzero eigenvalue of $X^{*}X$ and let
\[
W=Xf(X^{*}X).
\]
As always, $W$ is a partial isometry and
$X=WP$ for $P=\left(X^{*}X\right)^{\frac{1}{2}}=\left|X\right|$.
Now we discover 
\[
W^{\mathrm{T}}=f\left(X^{\mathrm{T}}\overline{X}\right)X^{\mathrm{T}}=f\left(XX^{*}\right)X=Xf\left(X^{*}X\right)=W.
\]
To create a unitary $U$ with $X=U\left|X\right|$ we must extend
$W$ to map $\ker(W)$ to $\ker\left(W^{*}\right)$. We can arrange
$U^{\mathrm{T}}=U$ as follows. Let $\mathbf{v}_{1},\dots,\mathbf{v}_{m}$
be an orthonormal basis of $\ker(W)$. Then $\overline{\mathbf{v}_{1}},\dots,\overline{\mathbf{v}_{m}}$
will be an orthonormal basis of $\ker(W^{*})=\ker(\overline{W})$
and we define $V$ to be zero on $\left(\ker(W)\right)^{\perp}$ and
$V\mathbf{v}_{j}=\overline{\mathbf{v}_{j}}.$ Thus $V^{*}$ will be
zero on $\left(\ker(W)\right)^{\perp}$ and 
$V^{*}\overline{\mathbf{v}_{j}}=\mathbf{v}_{j}$,
but the same can be said about $\overline{V}$. That means $V^{\mathrm{T}}=V$
and so $U=W+V$ will be the required symmetric unitary.

Notice there is no structure on $\left|X\right|$, but considered
with $\left|X^{*}\right|$ we get the formula
\begin{equation}
\left|X^{*}\right|=\left|X\right|^{\mathrm{T}}.
\label{eq:TwoAbsoluteValuesRelated}
\end{equation}
\end{proof}

For the self-dual situation, we shall see that a similar
construction works so long as we respect Kramers degeneracy.

\begin{prop}
\label{pro:self-dual-PartialIsomDouble}
If a partial isometry $W$
in $\mathbf{M}_{2N}(\mathbb{C})$ is self-dual, then the initial space
of $W$ will have even dimension and $\mathcal{T}$ will map the initial
space isometrically onto the final space of $W$.
\end{prop}

\begin{proof}
Applying Lemma~\ref{lem:left-right-Lemma} to the self-adjoint matrix
$W^{*}W$ we find
\begin{align*}
W^{*}W\mathbf{v}=\mathbf{v} & \implies\left(W^{*}W\right)^{\sharp}\mathcal{T}\mathbf{v}=\mathcal{T}\mathbf{v}\\
 & \implies\left(WW^{*}\right)\mathcal{T}\mathbf{v}=\mathcal{T}\mathbf{v}
\end{align*}
That is, when $\mathbf{v}\in\left(\ker W\right)^{\perp}$ we have
$\mathcal{T}\mathbf{v}\in\left(\ker W^{*}\right)^{\perp}$ and
$\left\Vert W^{*}\mathcal{T}\mathbf{v}\right\Vert =\left\Vert \mathbf{v}\right\Vert $.
This is useful because 
\begin{align*}
\left\langle \mathbf{v},W^{*}\mathcal{T}\mathbf{v}\right\rangle  
 & =\left\langle \mathbf{v},-W^{*}Z\overline{\mathbf{v}}\right\rangle \\
 & =\left\langle -W^{\mathrm{T}}Z\mathbf{v},\overline{\mathbf{v}}\right\rangle \\
 & =\left\langle -ZW\mathbf{v},\overline{\mathbf{v}}\right\rangle \\
 & =\left\langle \mathbf{v},W^{*}Z\overline{\mathbf{v}}\right\rangle \\
 & =-\left\langle \mathbf{v},W^{*}\mathcal{T}\mathbf{v}\right\rangle 
\end{align*}
which means $\mathbf{v}$ and $W^{*}\mathcal{T}\mathbf{v}$ are orthogonal,
and
\begin{align*}
W^{*}\mathcal{T}W^{*}\mathcal{T}\mathbf{v} & =W^{*}Z\overline{W^{*}Z\overline{\mathbf{v}}}\\
 & =W^{*}ZW^{\mathrm{T}}Z\mathbf{v}\\
 & =-W^{*}W^{\sharp}\mathbf{v}\\
 & =-W^{*}W\mathbf{v}\\
 & =-\mathbf{v}.
\end{align*}
If we start with a unit vector $\mathbf{v}$ in $\left(\ker W\right)^{\perp}$
then we end up with an orthogonal pair of unit vectors $\mathbf{v}$
and $\mathbf{w}$ with $W^{*}\mathcal{T}\mathbf{v}=\mathbf{w}$ and
$W^{*}\mathcal{T}\mathbf{w}=-\mathbf{v}$. 

If $\mathbf{q}$ is a vector in $\left(\ker W\right)^{\perp}$ that
is orthogonal to both $\mathbf{v}$ and $\mathbf{w}$ then $W^{*}\mathcal{T}\mathbf{q}$
will be orthogonal to both $W^{*}\mathcal{T}\mathbf{v}$ and $W^{*}\mathcal{T}\mathbf{w}$
since $\mathcal{T}$ preserves orthogonality everywhere and $W^{*}$
preserves the orthogonality of vectors in $\left(\ker W^{*}\right)^{\perp}.$
Thus we can create a basis of $\left(\ker W\right)^{\perp}$ out of
pairs 
\[
\mathbf{v}_{1},W^{*}\mathcal{T}\mathbf{v}_{1},\dots,\mathbf{v}_{m},W^{*}\mathcal{T}\mathbf{v}_{m}.
\]
\end{proof}

We need some examples of self-dual partial isometries. Treating vectors
as $2N$-by-$1$ matrices, if we set
\[
V=\left(\mathcal{T}\mathbf{v}\right)\mathbf{w}^{*}-\left(\mathcal{T}\mathbf{w}\right)\mathbf{v}^{*}
\]
then this rank two (at most) matrix is self-dual since
\begin{align*}
V^{\sharp} & =-Z\left(-Z\overline{\mathbf{v}}\mathbf{w}^{*}+Z\overline{\mathbf{w}}\mathbf{v}^{*}\right)^{\mathrm{T}}Z\\
 & =-Z\left(\overline{\mathbf{w}}\mathbf{v}^{*}Z-\overline{\mathbf{v}}\mathbf{w}^{*}Z\right)Z\\
 & =Z\overline{\mathbf{w}}\mathbf{v}^{*}-Z\overline{\mathbf{v}}\mathbf{w}^{*}\\
 & =-\mathcal{T}(\mathbf{w})\mathbf{v}^{*}+\mathcal{T}(\mathbf{v})\mathbf{w}^{*}\\
 & =V.
\end{align*}
If we start with $\mathbf{v}$ and $\mathbf{w}$ orthogonal, then
$V$ will be the rank-two partial isometry taking $\mathbf{v}$ to
$\mathcal{T}\mathcal{\mathbf{w}}$ and $\mathbf{w}$ to $\mathcal{T}\mathcal{\mathbf{v}}$.
We record this as a lemma that avoids the ugly notation.

\begin{lem}
\label{lem:KramersPartialIsometries}
Suppose $\mathbf{v}$ and \textbf{$\mathbf{w}$}
are orthogonal unit vectors. Then the partial isometry from $\mathbb{C}\mathbf{v}+\mathbb{C}\mathbf{w}$
to $\mathbb{C}\mathcal{T}\mathbf{v}+\mathbb{C}\mathcal{T}\mathbf{w}$
that sends $\mathbf{v}$ to $\mathcal{T}\mathbf{w}$ and $\mathbf{w}$
to $-\mathcal{T}\mathbf{v}$ will be self-dual.
\end{lem}

\begin{thm}
\label{thm:selfDualPolar} If $X$ in $\mathbf{M}_{2n}(\mathbb{C})$
satisfies $X^{\mathrm{\sharp}}=X$ then there is a unitary $U$ so
that $U^{\mathrm{\sharp}}=U$ and
$X=U\left|X\right|$.
\end{thm}

\begin{proof}
Once more $ W=Xf(X^{*}X) $
works to create a partial isometry $W$ with $X=WP$ for $P=\left|X\right|$
and this time we have 
$ W^{\sharp}=W$.
We need a partial isometry from $\ker(W)$ to $\ker\left(W^{*}\right)$
that is self-dual. The dimension of $\ker(W)$ will be even, by 
Lemma~\ref{pro:self-dual-PartialIsomDouble}.
Moreover if 
\[
\mathbf{v}_{1},\mathbf{w}_{1},\dots,\mathbf{v}_{m},\mathbf{w}_{m}
\]
is an orthogonal basis for $\ker(W)$ then 
\[
\mathcal{T}\mathbf{v}_{1},\mathcal{T}\mathbf{w}_{1},\dots,\mathcal{T}\mathbf{v}_{m},\mathcal{T}\mathbf{w}_{m}
\]
 will be an orthogonal basis for $\ker\left(W^{*}\right)$. The needed
self-dual partial isometry $V$ will send $\mathbf{v}_{j}$ to $\mathcal{T}\mathbf{w}_{j}$
and $\mathbf{w}_{j}$ to $\mathcal{T}\mathbf{v}_{j}$ and the self-dual
unitary we use will be $U=W+V$.
\end{proof}

\section{The odd particle causes even degeneracy}

We hope not to scare the mathematical reader with more discussion
of Kramers degeneracy. What Kramers discovered was that for a certain
systems involving a odd number of electrons, the Hamiltonian always
had all eigenvalues with even multiplicity. 

A mathematical manifestation of this is that the tensor product of
two dual operations is the transpose operation in disguise. In contrast
to that, the tensor product of three dual operations is a large dual
operation in disguise. Specifically if we implement an orthogonal
change of basis, the dual operation becomes the transpose.

This is essentially the same as the facts that is more familiar to
mathematicians, that $\mathbb{H}\otimes_{\mathbb{R}}\mathbb{H}\cong\mathbf{M}_{4}(\mathbb{R})$
and $\mathbb{H}\otimes_{\mathbb{R}}\mathbb{H}\otimes_{\mathbb{R}}\mathbb{H}\cong\mathbf{M}_{4}(\mathbb{H})$. 

In the following, we let $Z_{N}$ be as in (\ref{eq:Define-Z}), the
matrix that specified the dual operation.
\begin{lem}
Consider 
\[
U=\frac{1}{\sqrt{2}}\left(I\otimes I-iZ_{N}\otimes Z_{M}\right).
\]
For all $X\in\mathbf{M}_{2N}(C)$ and $Y\in\mathbf{M}_{2M}(\mathbb{C})$,
\[
U^{*}\left(X^{\sharp}\otimes Y^{\sharp}\right)U=\left(U^{*}\left(X\otimes Y\right)U\right)^{\mathrm{T}}.\]
\end{lem}

\begin{proof}
Since $Z_{K}^{\mathrm{T}}=-Z_{K}$ we see $U=U^{\mathrm{T}}$. Also
\begin{align*}
U^{*}U 
 & =\frac{1}{2}\left(I\otimes I+iZ_{N}\otimes Z_{M}\right)\left(I\otimes I-iZ_{N}\otimes Z_{M}\right)\\
 & =\frac{1}{2}\left(I\otimes I+\left(Z_{N}\otimes Z_{M}\right)^{2}\right)=I\otimes I
 \end{align*}
so $U$ is a unitary and $\overline{U}=U^{-1}$.

Since $U^{2}= -iZ_{N}\otimes Z_{M} $
we find
\begin{align*}
U^{*}\left(X^{\sharp}\otimes Y^{\sharp}\right)U & =\overline{U}\left(X^{\sharp}\otimes Y^{\sharp}\right)U\\
 & =\overline{U}\left(-iZ_{N}\otimes Z_{M}\right)\left(X^{\mathrm{T}}\otimes Y^{\mathrm{T}}\right)\left(iZ_{N}\otimes Z_{M}\right)U\\
 & =\overline{U}U^{2}\left(X^{\mathrm{T}}\otimes Y^{\mathrm{T}}\right)\overline{U}^{2}U\\
 & =\left(U^{*}\left(X\otimes Y\right)U\right)^{\mathrm{T}}.
\end{align*}
\end{proof}

We have refrained from discussing $C^{*}$-algebras, but here they
add clarity. What the lemma is showing is
\[
\left(\mathbf{M}_{2N}(\mathbb{C})\otimes\mathbf{M}_{2M}(\mathbb{C}),\sharp\otimes\sharp\right)
\cong
\left(\mathbf{M}_{2(M+N)}(\mathbb{C}),\mathrm{T}\right).
\]

\begin{lem}
For all $X\in\mathbf{M}_{2N}(C)$ and $Y\in\mathbf{M}_{2M}(\mathbb{C})$,
\[
X^{\mathrm{T}}\otimes Y^{\sharp}=\left(X\otimes Y\right)^{\mathrm{\sharp}}
\]
where the $\sharp$ on the right is taken with respect to $I\otimes Z_{M}$.
\end{lem}

\begin{proof}
This is much simpler, as 
\begin{align*}
\left(X\otimes Y\right)^{\mathrm{\sharp}}
 & =-\left(I\otimes Z_{M}\right)\left(X\otimes Y\right)^{\mathrm{\mathrm{T}}}\left(I\otimes Z_{M}\right)\\
 & =-\left(I\otimes Z_{M}\right)\left(X^{\mathrm{T}}\otimes Y^{\mathrm{T}}\right)\left(I\otimes Z_{M}\right)\\
 & =X^{\mathrm{T}}\otimes\left(-Z_{M}Y^{\mathrm{T}}Z_{M}\right).
\end{align*}
\end{proof}

Together, these lemmas show us that a tensor product
of three $\left(\mathbf{M}_{2N_{j}}(\mathbb{C}),\sharp\right)$
leads to something isomorphic to 
\[
\left(\mathbf{M}_{2(N_{1}+N_{2}+N_{3})}(\mathbb{C}),\sharp\right).
\]

\section{Acknowledgements}
The author gratefully
aknowleges the guidance and assistence received from mathematicians
and physicists, in particular Vageli
Coutsias, Matthew Hastings and Adam S\o rensen. 

This work was partially supported by a grant from the Simons Foundation
(208723 to Loring), and by the Efroymson fund at the University of New
Mexico.

\end{document}